\documentclass[10pt]{amsart}

\usepackage{amssymb,amsmath,amsthm,amsfonts, mathtools}
\usepackage{pdfsync}

\newcommand{\comment}[1]{}

\newtheorem{lem}{Lemma}[section]

\newtheorem{propn}[lem]{Proposition}
\newtheorem{cor}[lem]{Corollary}
\newtheorem{thm}[lem]{Theorem}

\theoremstyle{remark}

\theoremstyle{definition}

\newcommand{\R}{\mathbb R}

\newcommand{\Z}{\mathbb Z}

\newcommand{\N}{\mathbb N}
\newcommand{\C}{\mathbb C}

\newcommand{\D}{\delta}

\newcommand{\VE}{\varepsilon}
\newcommand{\A}{\alpha}

\newcommand{\eps}{\ensuremath{\varepsilon}}

\newcommand{\La}{\ensuremath{\Lambda}}
\newcommand{\Te}{\ensuremath{\Theta}}
\newcommand{\al}{\ensuremath{\alpha}}
\newcommand{\be}{\ensuremath{\beta}}

\newcommand{\te}{\ensuremath{\theta}}
\newcommand{\de}{\ensuremath{\delta}}

\newcommand{\MM}{\ensuremath{\mathcal{M}}}

\newcommand{\subs}{\ensuremath{\subseteq}}

\newcommand{\eq}{\begin{equation}}
\newcommand{\ee}{\end{equation}}

\begin{document}
\title{Simultaneous Polynomial Recurrence}
\author{Neil Lyall\quad\quad\quad\'Akos Magyar}
\thanks{Both authors were partially supported by NSF grants.}

\address{Department of Mathematics, The University of Georgia, Athens, GA 30602, USA}
\email{lyall@math.uga.edu}
\address{Department of Mathematics, University of British Columbia, Vancouver, B.C. V6T 1Z2, Canada}
\email{magyar@math.ubc.ca}


\begin{abstract}
Let $A\subseteq\{1,\dots,N\}$ and $P_1,\ldots,P_\ell\in\Z[n]$ with $P_i(0)=0$ and $\deg P_i=k$ for every $1\leq i\leq\ell$.

We show, using Fourier analytic techniques, that for every $\VE>0$, there necessarily exists $n\in\N$ such that
\[\frac{|A\cap (A+P_i(n))|}{N}>\left(\frac{|A|}{N}\right)^2-\VE\] holds simultaneously for $1\leq i\leq \ell$ (in other words all of the polynomial shifts of the set $A$ intersect $A$ ``$\VE$-optimally"), as long as $N\geq N_1(\VE,P_1,\ldots,P_\ell)$.
The quantitative bounds obtained for $N_1$ are explicit but poor; we establish that $N_1$ may be taken to be a constant (depending only on $P_1,\dots,P_\ell$) times a tower of 2's of height $C_{k,\ell}^*+C\eps^{-2}$.

\end{abstract}
\maketitle

\setlength{\parskip}{4pt}


\section{Introduction.}
\subsection{Background}
The study of recurrence properties of dynamical systems goes back to the beginnings of ergodic theory.
 If $A$ is a measurable subset of a probability space $(X,\mathcal{M},\mu)$ with $\mu(A)>0$ and $T$ is a measure preserving transformation, then it was already shown by Poincar\'e \cite{P} that $\mu(A\cap T^{-n}A)>0$ for some  natural number $n$ (and hence for infinitely many).
This result was subsequently sharpened by Khintchine \cite{K}, who observed that for every $\eps>0$ there in fact exist $n\in\N$ such that $\mu(A\cap T^{-n}A)>\mu(A)^2-\eps$.
Note that in general this lower bound is sharp, since $\mu(A\cap T^{-n} A)\to\mu(A)^2$ as $n\to \infty$ whenever $T$ is a mixing transformation.

A polynomial version of Khintchine's result, where the set of natural numbers $n$ is replaced by the values of an integral polynomial $P(n)$ that satisfies $P(0)=0$, was established by Furstenberg \cite{F}, for a proof see also \cite{M} or \cite{B}. Recently, far reaching generalizations of Furstenberg's result have been obtained in the settings of multiple recurrence: let $(X,\MM,\mu,T)$ be an invertible measure preserving system, $A\in\mathcal{M}$ and $P_1,\ldots,P_\ell$ be any linearly independent family of integral polynomials with $P_i(0)=0$ for all $1\leq i\leq \ell$, then Frantzikinakis and Kra \cite{FK} have shown that for any $\VE>0$, there necessarily exists $n\in \N$ such that
\eq\label{FK}
\mu(A\cap T^{-P_1(n)}A \cap\cdots\cap T^{-P_\ell(n)}A)>\mu(A)^{\ell+1}-\eps.\ee

We note that it follows from an earlier counterexample of Ruzsa \cite{BHK} that this result cannot hold in general for dependent polynomials when $\ell\geq2$ nor even in the setting of ergodic systems when $\ell\geq4$.  Bergelson, Host and Kra established in \cite{BHK} that (\ref{FK}) does hold, under this additional assumption that $T$ is ergodic in the case of (dependent) linear polynomials when $\ell=2,3$. Frantzikinakis \cite{NF} has investigated the situation for higher degree polynomials.

These multiple recurrence results contrast sharply with the situation when one drops the requirement that the measure of the intersections in (\ref{FK}) are ``optimally'' large: the Polynomial Szemer\'edi Theorem of Bergelson and Leibman \cite{BL} states that  if
$(X,\MM,\mu,T)$ is an invertible measure preserving system and $P_1,\ldots,P_\ell\in\Z[n]$ with $P_i(0)=0$ for all $1\leq i\leq \ell$, then for any $A\in\mathcal{M}$ with $\mu(A)>0$ there necessarily exists $n\in \N$ such that
\eq\label{BL}
\mu(A\cap T^{-P_1(n)}A \cap\cdots\cap T^{-P_\ell(n)}A)>0.\ee
Note that the case when all the polynomials are linear corresponds to Furstenberg's Multiple Recurrence Theorem \cite{F}, which is, via Furstenberg's correspondence principle, equivalent to Szemer\'edi's Theorem on arithmetic progressions.

For a comprehensive survey of the impact of the Poincar\'e recurrence principle in ergodic theory, especially as pertains to the field of ergodic Ramsey theory/additive combinatorics, see \cite{FM}, \cite{B2} and \cite{Kra}.

\subsection{Statement of Main Results}

In this article we will concern ourselves with the study of simultaneous (single) polynomial recurrence.
The following result gives a full generalization of Furstenberg's result in this direction and can be established, as we shall see below, using current and well-known methods in ergodic Ramsey theory.

\begin{thm}\label{ergodic}
Let $(X,\MM,\mu,T)$ be an invertible measure preserving system, $P_1,\ldots,P_\ell \in\Z[n]$ with $P_i(0)=0$ for all $1\leq i\leq \ell$ and $A\in\mathcal{M}$.
For every $\eps>0$, there exists $n\in\N$  such that
\eq
\mu(A\cap T^{-P_i(n)} A)>\mu(A)^2-\eps\ \ \ \ \text{for all}\ \ \ \ 1\leq i\leq \ell.
\ee
\end{thm}

Note that there are no assumptions that the polynomials in Theorem \ref{ergodic} are linearly independent.

In the special case when $k=1$, that is when all of the polynomials are linear, this result can be established using only combinatorial methods and the following quantitative result can be obtained.

\begin{thm}[Griesmer \cite{John}]\label{John}
Let $(X,\MM,\mu,T)$ be an invertible measure preserving system, $c_1,\dots,c_\ell\in\Z\setminus\{0\}$ with $\ell\leq2^m$ for some $m\in\N$  and $A\in\mathcal{M}$.
For every $\VE>0$ and $B\subseteq\N$ with
$\log^{m}|B|\geq C\VE^{-1}$, there exists a non-zero $n\in B-B$ such that
\eq
\mu(A\cap T^{-c_in} A)>\mu(A)^2-\eps\ \ \ \ \text{for all}\ \ \ \ 1\leq i\leq \ell.
\ee
\end{thm}

It follows from a variant of Furstenberg's correspondence principle, see Frantzikinakis and Kra \cite{FK} (in particular the proof of Theorem 2.2), that Theorem \ref{ergodic} has the following combinatorial consequence.

\begin{cor}\label{cor}
Let $P_1,\ldots,P_\ell \in\Z[n]$ with $P_i(0)=0$ for all $1\leq i\leq \ell$.
For every $\VE>0$ there exists $N_0=N_0(\eps,P_1,\dots,P_\ell)$ such that if $N\geq N_0$ and $A\subs [1,N]$, then there exists $n\in\N$ such that
\eq
\frac{|A \cap(A+P_i(n))|}{N} >\left(\frac{|A|}{N}\right)^2-\VE\ \ \ \ \text{for all}\ \ \ 1\leq i\leq \ell.\ee
\end{cor}

We note that this correspondence give no quantitative bounds in the finite setting of Corollary \ref{cor} (other than the special case when all of the polynomials are linear).
However,
if we relax the requirement that the intersections are ``optimally'' large and ask merely that they are non-empty then one has the following result.

\begin{thm}[Lyall and Magyar \cite{LM1'}]\label{LM1}
Let $0<\D<1$ and $P_1,\ldots,P_\ell \in\Z[n]$ with $P_i(0)=0$ and $\deg P_i\leq k$ \emph{for all} $1\leq i\leq \ell$. There exists a constant $C=C(P_1,\dots,P_\ell)$ such that if $N\geq \exp(C\D^{-\ell(k-1)}\log\D^{-1})$ and $A\subseteq[1,N]$ with $|A|\geq \D N$, then there exists $n\in\N$ for which
\eq
A\cap(A+P_i(n))\neq\emptyset\ \ \ \ \text{for all}\ \ \ 1\leq i\leq \ell.
\ee
\end{thm}

While if we continue to insist on ``optimally'' large intersections, but restrict ourselves to the case $\ell=1$, namely the case of a single polynomial, then we have the following result.

\begin{thm}[Lyall and Magyar \cite{LM2}]\label{LM2}
Let $A\subseteq[1,N]$, $P(n)\in\Z[n]$ with $P(0)=0$ and $\VE>0$. There exists a constant $C=C(P)$ such that if $N\geq\exp\exp(C\VE^{-1}\log\VE^{-1})$, then there exists $n\in\N$ for which
\eq
\dfrac{|A\cap(A+P(n))|}{N}>\left(\dfrac{|A|}{N}\right)^2-\VE.
\ee
\end{thm}

The main objective of the present paper is to present the proof of a (partial) common generalization of Theorems \ref{LM1} and \ref{LM2}. To be more precise, our objective is to establish, using Fourier analytic methods, Corollary \ref{cor} with explicit quantitative bounds, in the special case when all of the polynomials are of the \emph{same} degree. In particular we are able to establish the following.

\begin{thm}\label{MainThm}
Let $P_1,\ldots,P_\ell \in\Z[n]$ with $P_i(0)=0$ and $\deg P_i=k$ for all $1\leq i\leq \ell$. For every $\VE>0$ there exists $N_1= N_1(\eps,P_1,\dots,P_\ell)$ such that if $N\geq N_1$ and $A\subs [1,N]$, then there exist $n\in\N$ such that
\eq\label{MainEq}
\frac{|A \cap(A+P_i(n))|}{N} >\left(\frac{|A|}{N}\right)^2-\VE\ \ \ \ \text{for all}\ \ \ 1\leq i\leq \ell.\ee
In particular, the number $N_1(\eps,P_1,\dots,P_\ell)$ may be taken to be a constant (depending only on $P_1,\dots,P_\ell$) times a tower of 2's of height $C_{k,\ell}^*+C\eps^{-2}$.
\end{thm}

As the reader will no doubt have noticed, the bounds obtained for $N_1(\eps,P_1,\dots,P_\ell)$, while explicit, are rather poor. It is our belief that these are far from the best bounds possible whose dependence on $\eps$ we would expect to be at least of exponential type.

We also note that, because of the introduction of different scales (specifically in Lemma \ref{FourierError}), our current Fourier analytic approach appears to be insufficient for the task establish a quantitative result along the lines of Theorem \ref{MainThm} for polynomials with different degrees. In particular, we are not aware of any quantitative result of this type even in the simplest case, namely $\ell=2$ with $P_1(n)=n$ and $P_2(n)=n^2$.

\subsection{An outline of the paper}\label{outline}

As we have been unable to find a proof of Theorem \ref{ergodic} in the literature, we give a complete proof of this result in Section \ref{ergsec}. We feel that the inclusion of this argument will also help illuminate for the reader the proof of our main result, namely Theorem \ref{MainThm}, which we present in Section \ref{fousec}.

In Section \ref{johnsec} we communicate an elegant combinatorial proof of Theorem \ref{John}, using Ramsey's theorem, that was shown to us by John Griesmer. We are grateful to John for both showing us this argument and giving us his permission to include it here.

Theorem \ref{LM1} was first established in \cite{LM1}, but only in the case of linearly independent polynomials. In Section \ref{liftsec} we include a  simple modification of the \emph{lifting} argument used in \cite{LM1}, to extend the original result in \cite{LM1} to the case of linearly dependent polynomials, thus establishing Theorem \ref{LM1} as stated above and in \cite{LM1'}.

Finally, we also include a short appendix on counting solutions to systems of polynomial diophantine equations as well as a somewhat lengthier appendix on simultaneous polynomial diophantine approximation.

\subsection{Notational convention}
Throughout the paper the letters $c$, $C$ will denote absolute
constants. These constants will generally satisfy $0<c\ll1\ll C$. Different instances of the notation, even on the same line,
will typically denote different constants.


\section{The proof of Theorem \ref{ergodic}}\label{ergsec}

Let $\VE>0$ and Let $(X,\MM,\mu,T)$ be an invertible measure preserving system.

We define $U_Tf(x):=f(Tx)$ and note that $U_T$ then defines a unitary operator on the Hilbert space of all square integrable function $L^2(X,\mu)$.
If we define $f=1_A$, then
\eq\label{intersection}
\mu(A\cap T^{-P_i(n)}A)=\langle f,U_T^{P_i(n)}f \rangle
\ee
for each $1\leq i\leq \ell$.

\subsection{Decomposition}
We now proceed by decomposing $f$ into an almost periodic (structured) component and a weakly-mixing (anti-structured) component, the so-called Koopman-von Neumann decomposition.

\begin{propn}[Koopman and von Neumann \cite{KvN}, see also \cite{M}]\label{KvN}
Let $\mathcal{H}:=L^2(X,\mu)$, then
\eq
\mathcal{H}=\mathcal{H}_c\oplus \mathcal{H}_{wm}\ee
where
\eq
\mathcal{H}_c=\left\{f\in \mathcal{H}\,:\, \left\{U_Tf\,:\,n\in\Z\right\} \ \text{is pre-compact}\right\}\ee
and
\eq
\mathcal{H}_{wm}=\left\{f\in \mathcal{H}\,:\, \lim_{N\to\infty}\frac{1}{N}\left|\left\{1\leq n\leq N\,:\,\left|\langle U_T^n f,g\rangle\right|\geq\VE\right\}\right|=0 \ \text{for all $\VE>0$ and $g\in \mathcal{H}$}\right\}.
\ee
Moreover, \[\mathcal{H}_{wm}=\mathcal{H}_{c}^\perp.\]
\end{propn}

Using Proposition \ref{KvN} we can therefore uniquely decompose
\eq
f=f_1+f_2,
\ee
with $f_1\in\mathcal{H}_{c}$ and $f_2\in\mathcal{H}_{wm}$. Moreover, it is easy to see that these functions also enjoy the property that $0\leq f_1\leq 1$ and  $\langle f_2,1\rangle=0$.
 We note that it follows immediately from the Cauchy-Schwarz inequality that
\eq
\mu(A)^2=\langle f,1\rangle^2=\langle f_1,1\rangle^2\leq\langle f_1,f_1\rangle.
\ee

\subsection{Proof of Theorem \ref{ergodic}}

Inserting our decomposition $f=f_1+f_2$ into (\ref{intersection}) we see that
\eq\label{inserting}
\mu(A\cap T^{-P_i(n)}A)=\langle f_1,U_T^{P_i(n)}f_1 \rangle+\langle f,U_T^{P_i(n)}f_2 \rangle+\langle U_T^{-P_i(n)}f_2,f_1 \rangle
\ee
for each $1\leq i\leq \ell$.

Our stategy to prove Theorem \ref{ergodic} will be to show that
$U_T^{P_i(n)} f_1\approx f_1$
and hence
\[\langle U_T^{P_i(n)}f_1,f_1\rangle\approx\langle f_1,f_1\rangle\geq\mu(A)^2\]
simultaneously for all $1\leq i\leq \ell$ for a positive proportion of $1\leq n\leq N$,
while for any given $g\in L^2(X,\mu)$ and $P\in\Z[n]$ the proportion of $1\leq n\leq N$ for which $\langle g,U_T^{P(n)}f_2 \rangle\approx0$ tends to $1$ as $N\to\infty$.

More precisely, we will establish the following two lemmas from which Theorem \ref{ergodic} follows immediately.

\begin{lem}[Main term estimate]\label{MainErgodic}
There exists $c_0=c_0(\eps,k,\ell, f_1)>0$ such that for all large $N$
\eq
\left|\left\{1\leq n\leq N\,:\,\left\|U_T^{P_i(n)}f_1-f_1\right\|\leq\eps/2  \text{ for all } 1\leq i\leq \ell\right\}\right|\geq c_0 N.
\ee
\end{lem}

\begin{lem}[Error term estimate]\label{ErrorErgodic}
Let $P\in\Z[n]$ with $P(0)=0$ and $g\in L^2(X,\mu)$, then
\eq\label{Dlim}
\lim_{N\to\infty}\frac{1}{N}\left|\left\{1\leq n\leq N\,:\,|\langle U_T^{P(n)}f_2,g\rangle|\geq\VE/4\right\}\right|=0.
\ee
\end{lem}

Indeed, if $N$ is large enough then from Lemma \ref{MainErgodic} it follows that there must exist at least $c_0N$ values of $n\in [1,N]$ for which $\ \|U_T^{P_i(n)}f_1-f_1\|\leq\eps/2$ 
and hence
\[\langle f_1, U_T^{P_i(n)}f_1\rangle\geq\langle f_1,f_1\rangle-|\langle f_1, U_T^{P_i(n)}f_1-f_1\rangle|\geq\mu(A)^2-\eps/2\]
simultaneously for all $1\leq i\leq \ell$.
While from Lemma \ref{ErrorErgodic} it follows that if $N$ is taken sufficiently large then the absolute value of the last two error terms in (\ref{inserting}) can be made less than $\eps/4$ for all but at most $c_0N/2$ values of $n\in [1,N]$ simultaneously for $1\leq i\leq \ell$.
Thus for any $\VE>0$ there exists $n\in [1,N]$ (in fact a positive proportion) for which
\[\mu(A\cap T^{-P_i(n)}A)>\mu(A)^2-\eps\]
for all $1\leq i\leq \ell$. This completes the proof of Theorem \ref{ergodic}.\qed

\subsection{Proof of Lemmas \ref{MainErgodic} and  \ref{ErrorErgodic}}

The proof of Lemma \ref{MainErgodic} is based on van der Waerden's theorem and the magical identity
\eq\label{magic}
\sum_{t=0}^j (x+td)^j {j \choose t}\,(-1)^{j-t}\ =\ j!\,d^j
\ee
the validity of which can be easily verified for all $j\in\N$ by induction.
Lemma \ref{MainErgodic} is of course in essence a result on simultaneous diophantine approximation and the proof we present below is essentially an adaptation of the proof of Proposition 1.5 (on quadratic recurrence) in \cite{T}.

The proof of Lemma \ref{ErrorErgodic} follows from the Hilbert space version of van der Corput's Lemma (for a statement of this version see either \cite{B} or \cite{M}) and is well-known, but for the sake of completeness we have chosen to also sketch its proof below.

\begin{proof}[Proof of Lemma \ref{MainErgodic}]
Let $P_i(x)=\sum_{j=1}^k c_{ij} x^j$, and let $\eta=\eta(\eps,k)>0$ be a small constant to be chosen later.

Cover the orbit $\{T^n f\,:\, n\in\Z\}$ by balls $B_1,\ldots,B_M$ of diameter $\eta$ and use this to define a (matrix-valued) coloring $\chi:\Z\to [1,M]^{k\ell}$ of the integers by setting, for each $1\leq i\leq \ell$ and $1\leq j\leq k$, $\chi_{ij}(n)=r$ if $U_T^{c_{ij}n^j} f_1\in B_r$.

By (the averaged version of) van der Waerden's theorem there is a constant $c_1=c_1(M,k,\ell)$
such that the number of monochromatic $(k+1)$-term arithmetic progressions in $[1,N]$ is at least $c_1 N^2$, provided $N$ is sufficiently large.
Note that this implies that there will be at least $c_1 N$ monochromatic $(k+1)$-term arithmetic progressions in $[1,N]$ with \emph{different} step sizes $d$.

Let $d$ be the step size of a monochromatic arithmetic progression $\{x+td:\ 0\leq t\leq k\}$ with respect to the coloring $\chi$ in $[1,N]$. Then for any fixed $i,j$ one has
\[\|U_T^{c_{ij}(x+td)^j} f_1- U_T^{c_{ij}x^j} f_1\|\leq\eta\]
for all $0\leq t\leq k$,
thus by (\ref{magic}) it follows that
\begin{align*}
\left\|U_T^{c_{ij}\,j!d^j} f_1-f_1\right\|&=\left\|U_T^{c_{ij}\sum_{t=0}^j(x+td)^j {j \choose t}(-1)^{j-t}} f_1-f_1\right\|\\
&=\left\|U_T^{c_{ij}\sum_{t=0}^j(x+td)^j {j \choose t}(-1)^{j-t}} f_1-U_T^{c_{ij}\sum_{t=0}^j x^j {j \choose t}(-1)^{j-t}} f_1\right\|\\
&\leq \sum_{t=0}^j {j \choose t}\left\|U_T^{c_{ij}(x+td)^j}f_1-U_T^{c_{ij}x^j} f_1\right\|\\
&\leq2^j\eta.
\end{align*}

Here we have used the facts that $\|U_T^{m_1+m_2}f_1-U_T^{n_1+n_2}f_1\|\leq \|U_T^{m_1}f_1-U_T^{n_1}f_1\|+\|U_T^{m_2}f_1-U_T^{n_2}f_1\|$ and $\|U_T^{bn}f_1-U_T^{bm}f_1\|\leq |b|\,\|U_T^n f_1-U_T^m f_1\|$ which follows from the triangle inequality and the fact that $U_T$ is a unitary operator on $L^2(X,\mu)$.
Thus
\[\left\|U_T^{c_{ij}(k!d)^j}f_1-f_1\right\|\leq(k!^j/j!)\left\|U_T^{j!c_{ij}d^j}f_1-f_1\right\|\leq(k!)^k 2^k\eta.\]

Letting $n=k! d$ it follows that for all $1\leq i\leq \ell$ we have
\eq \left\|U_T^{P_i(n)}f_1-f_1\right\|\leq k(k!)^k 2^k\eta\leq\eps/2\ee
provided $\eta$ is chosen small enough.
Since the number of such $d\in [1,N/k!]$ is at least $(c_1/k!)N$, the lemma follows.
\end{proof}

\begin{proof}[Proof of Lemma \ref{ErrorErgodic}]
We give a proof by induction, using the fact that (\ref{Dlim}) is equivalent to
\eq\label{notDlim}
\lim_{N\to\infty}\frac{1}{N}\sum_{n=1}^N\left|\langle U_T^{P(n)}f_2,g\rangle\right|^2=0
\ee
and that when $\deg P=1$, that is $P(x)=mx$ for some $m\in\Z$, then the conclusion of the lemma is an immediately consequence of the weak-mixing properties of $f_2$.

Let $k\geq 1$, $P\in\Z[n]$ be a polynomial of degree $k+1$, and assume that (\ref{notDlim}) holds for all polynomials of degree at most $k$. We will show that (\ref{notDlim}) holds for all $g\in L^2(X,\mu)$. To this end we note that
\[\frac{1}{N}\sum_{n=1}^N\left|\langle U_T^{P(n)}f_2,g\rangle\right|^2=\frac{1}{N}\sum_{n=1}^N\left\langle (U_T\times U_T)^{P(n)}(f_2\times f_2),g\times g\right\rangle_{X\times X}\]
and hence that it suffices to show that
\[\lim_{N\to\infty}\left\|\frac{1}{N}\sum_{n=1}^N (U_T\times U_T)^{P(n)}(f_2\times f_2) \right\|_{L^2(X\times X)}=0.\]

Let $x_n:= (U_T\times U_T)^{P(n)}(f_2\times f_2)$ for $n\in\Z$, and let $h\in\Z\setminus\{0\}$. Since
\begin{align*}
\left\langle x_{n+h},x_{n}\right\rangle_{X\times X}&=\left\langle(U_T\times U_T)^{P(n+h)-P(n)-P(h)}(f_2\times f_2),(U_T\times U_T)^{-P(h)}(f_2\times f_2)\right\rangle_{X\times X}\\
&=\left|\left\langle U_T^{P(n+h)-P(n)-P(h)}f_2,U_T^{-P(h)}f_2\right\rangle\right|^2
\end{align*}
it follows from the inductive hypothesis, since the polynomial $P(n+h)-P(n)-P(h)$ has degree at most $k$ and $U_T^{-P(h)}f_2\in L^2(X,\mu)$, that
\[\lim_{N\to\infty}\frac{1}{N}\sum_{n=1}^N \left\langle x_{n+h},x_{n}\right\rangle_{X\times X}=0.\]
The claim now follows from the Hilbert space version of van der Corput's Lemma, see either \cite{B} or \cite{M}.
\end{proof}



\section{The proof of Theorem \ref{MainThm}}\label{fousec}

Let $\VE>0$ and $P_1,\dots,P_\ell\in\Z[n]$ with $P_i(0)=0$ and $\deg P_i=k$ \emph{for all} $1\leq i\leq \ell$.

\subsection{The Fourier transform, uniformity and polynomial shifts}
Let $\Z_N$ denote the group $\Z/N\Z$.
\subsubsection{The Fourier transform}
Given $f:\Z_N\to\C$ we define its (discrete) Fourier transform, $\widehat{f}:\widehat{\Z_N}\to\C$, by
\[\widehat{f}(\xi)=\frac{1}{N}\sum_{x\in\Z_N}f(x)e(-x\xi/N)\]
where $e(x)=e^{2\pi i x}$ and $\widehat{\Z_N}$ denote the \emph{dual group} of of all characters on $\Z_N$.

It is easy to see that we can identify $\Z_N$ with its dual. There are two natural measures that one can put on $\Z_N$, namely uniform probability measure and counting measure. As is customary, we shall use the uniform probability measure on $\Z_N$ and the counting measure on $\Z_N$ when it is being identified with its dual group. We then define $L^p$-norms and $\ell^p$-norms as follows.

We define $L^p$ to be the space of all functions from $\Z_N$ to $\C$, with the norm
\[\|f\|_p=\Bigl(\frac{1}{N}\sum_{x\in\Z_N}|f(x)|^p\Bigr)^{1/p},\]
where this is interpreted as $\max_{x\in\Z_N}|f(x)|$ when $p=\infty$. We define $\ell^p$ to be the space of all functions from $\widehat{\Z_N}$ to $\C$, with the norm
\[\|F\|_p=\Bigl(\sum_{\xi\in\Z_N}|F(\xi)|^p\Bigr)^{1/p},\]
where we again interpreted this as $\max_{\xi\in\Z_N}|F(\xi)|$ when $p=\infty$.

 In contrast to the situation for the Fourier transform on $\R$, the Fourier inversion formula and Plancherel's identity, namely
\[f(x)=\sum_{\xi\in\Z_N}\widehat{f}(\xi)e(x\xi/N)\quad\quad\text{and}\quad\quad
\|f\|_2=\|\widehat{f}\|_2\] are, in this setting, immediate and simple consequences of the familiar orthogonality relation
\[\frac{1}{N}\sum_{x\in\Z_N}e(x\xi/N)=\begin{cases}
1\quad\text{if \ $\xi=0$}\\
0\quad\text{if \ $\xi\ne0$}
\end{cases}.\]

\subsubsection{Uniformity and polynomial shifts}
We now fix a set $A\subseteq[1,N]$.
In order to use Fourier analytic techniques we will, as is customary, identify $[1,N]$ with $\Z_N$ and consider $A$ as a subset of $\Z_N$. In order to ensure that working in $\Z_N$ will not, in any essential way, affect the validity of (\ref{MainEq}), we will restrict our attention to those values of $n$ for which
\eq\label{M} 1\leq n\leq M:=c(\eps N)^{1/k}\ee
with $c=c(P_1,\dots,P_\ell)$ chosen sufficiently small such that $|P_i(n)|\leq \eps N$ for all $1\leq i\leq \ell$. Note that doing this we will ensure that the size of $A\cap (A+P_i(n))$ will increase by at most $\eps N$, due to overlapping when the shifts $P_i(n)$ take place in $\Z_N$, and as such working in $\Z_N$ will not affect
the validity of (\ref{MainEq}) (other than changing $\VE$ to $2\VE$).

Note that if we define $f=1_A$, then
\eq\label{intersectionF}
\frac{|A\cap(A+P_i(n))|}{N}=\frac{1}{N}\sum_{x\in\Z_N}f(x)f(x-P_i(n))
\ee
for each $1\leq i\leq \ell$.

It is well-known (and easy to verify, see Lemma 2.2 in \cite{gowers1}) that if $g,h:\Z_N\to\C$ with both $\|g\|_2$ and $\|h\|_2$ bounded by $1$, then $\|\widehat{h}\|_\infty\leq \eta$ is equivalent to
\[\frac{1}{N}\sum_{n=1}^N\left|\frac{1}{N}\sum_{x\in\Z_N}h(x)g(x-n)\right|^2\leq \eta\]
and consequently that
\[\left|\left\{1\leq n\leq N\,:\, \left|\frac{1}{N}\sum_{x\in\Z_N}h(x)g(x-n)\right|\geq \eta^{1/3}\right\}\right|\leq \eta^{1/3}N\]
whenever $h$ satisfies the uniformity assumption that $\|\widehat{h}\|_\infty\leq \eta$.

The main aim of this section is to show that this phenomenon continues to hold for each of the polynomial shifts $P_i(n)$ for almost all $n$ satisfying (\ref{M}).
In particular we show the following.

\begin{lem}[Error term estimate]\label{FourierError}
Let $\VE>0$, $P\in\Z[n]$ with $\deg P=k\geq2$ and $M=c(\eps N)^{1/k}$.

If $g,h:\Z_N\to\C$ with $\|g\|_2\leq 1$, $\|h\|_2\leq1$ and $\|\widehat{h}\|_\infty\leq\eta$ with $0<\eta\leq\eps$,
then
\eq
\left|\left\{1\leq n\leq M\,:\, \left|\frac{1}{N}\sum_{x\in\Z_N}h(x)g(x-P(n))\right|\geq \eta^{1/K}\right\}\right|\leq C_1\,\eta^{1/K}M
\ee
for any positive integer $K\geq C\,k^2\log k$, where then $C_1$ is a large constant depending only on $P$.
\end{lem}

\subsubsection{Theorem \ref{MainThm} for uniform sets}
We quickly remark that from Lemma \ref{FourierError} we can immediately deduce Theorem \ref{MainThm} in the special case of (suitably) uniform sets $A$.
Recall that a set $A$ is said to be $\eta$-uniform if $|\widehat{1_A}(\xi)|\leq\eta$ for all $\xi\in\Z_N\setminus\{0\}$, or equivalently $\|\widehat{f_A}\|_\infty\leq\eta$, where $f_A=1_A-|A|/N$ denotes the so-called balanced function of $A$.
Inserting the decomposition $f=f_A+|A|/N$ into (\ref{intersectionF}) and recalling that the function $f_A$ has mean value zero, we obtain the following corollary to Lemma \ref{FourierError}.

\begin{cor}\label{uniform} 
If $A\subseteq\Z_N$ is $\eta$-uniform with $0<\eta\leq\eps^K$ and $K\geq C\,k^2\log k$, then there necessarily exist at least $(1-\ell C_1\eta^{1/K})M$ values of $n\in[1,M]$ (and hence at least one if $\eta^{1/K}\leq\min\{\eps, 1/2\ell C_1\}$) for which
\[\left|\frac{|A\cap(A+P_i(n))|}{N}-\left(\frac{|A|}{N}\right)^2\right|<\VE\]
holds simultaneously for $1\leq i\leq \ell$.
\end{cor}

\subsubsection{Proof of Lemma \ref{FourierError}}

We make use of the following well-known fact from number theory which we state here, and whose proof will be given in Appendix A for the sake of
completeness.

\begin{propn}\label{nt} Let $P\in\Z[n]$ with $\deg P=k$ and
$K\geq C\,k^2 \log\,k$ be a natural number. Then for any $M\in\N$, the number of $2K$-tuples
$(n_1,\ldots,n_K,m_1,\ldots,m_K)\in [1,M]^{2K}$ satisfying
\eq\label{2.5}
P(n_1)+\cdots +P(n_K)=P(m_1)+\cdots +P(m_K)\ee
is bounded by $C_0 M^{2K-k}$, where $C_0$ is a large constant depending only on $P$.
\end{propn}

In order to prove Lemma \ref{FourierError}, with $4K$ in place of $K$, it is suffices to
show that
\eq\label{sum}
\sum_{n=1}^M \left|\frac{1}{N}\sum_{x\in\Z_N}h(x)g(x-P(n))\right|\leq C_1\eta^{1/2K}M.
\ee

In order to verify (\ref{sum}) we introduce a weight function $w:\Z_N\to \{-1,1\}$ such that the left side
of (\ref{sum}) may be written (using  Fourier inversion) as
\eq
\frac{1}{N}\sum_{n=1}^M\sum_{x\in\Z_N}h(x)g(x-P(n)) w(n)=
\sum_{\xi\in\Z_N}\widehat{h}(\xi)\overline{\widehat{g}(\xi)}S_w(\xi)
\ee
where
\eq
S_w(\xi)=\sum_{n=1}^M w(n)e(P(n)\xi/N).
\ee

Note that we have no control over the weight function $w(n)$ and
hence cannot hope for any non-trivial pointwise bound on $|S_w(\xi)|$, nevertheless we can obtain sharp estimates for the
higher moments of $S_w$ using Proposition \ref{nt}.
Indeed for any positive integer $K\geq
C\,k^2\,\log\,k$ one estimates
\begin{align*}
\|S_w\|_{2K}^{2K}&=\sum_{\xi\in\Z_N} |S_w(\xi)|^{2K}\\
&= \!\!\!\!\sum_{\substack{n_1,\dots,n_K =1\\ m_1\dots,m_K=1}}^M \!\!\!\!w(n_1)\cdots%
w(m_K)\sum_{\xi\in\Z_N} e\left(\left(P(n_1)+\cdots
+P(n_K)-P(m_1)-\cdots -P(m_K)\right)\xi/N\right)\\
&\leq N\left|\left\{(n_1,\dots,n_K,m_1\ldots,m_K)\in[1,M]^{2K}\,:\,
P(n_1)+\cdots +P(n_K)= P(m_1)+\cdots +P(m_K)\right\}\right|.
\end{align*}
Since $1\leq n_i,m_i\leq M$ and $M = c\,(\eps N)^{1/k}$ with a
sufficiently small constant $c$, the equality
\[P(n_1)+\cdots+P(n_K)= P(m_1)+\cdots +P(m_K)\]
holds in $\Z_N$ if and only it holds in $\Z$.

It therefore follows from Proposition \ref{nt} and H\"older's inequality that
\[\sum_{\xi\in\Z_N} |\widehat{h}(\xi)||\widehat{g}(\xi)||S_w(\xi)|\ \leq\ \|\widehat{g}\|_2\,
\|\widehat{f}\|_\frac{2K}{K-1}\, \|S_w\|_{2K} \leq \eta^\frac{1}{K}(C_0NM^{2K-k})^{1/2K}\leq C_1\eta^{1/2K}M\]
with $C_1=(C_0/c^k)^{1/2K}$
since
$\|\widehat{h}\|_\frac{2K}{K-1}\leq \|\widehat{h}\|_\infty^{1/K}\|\widehat{h}\|_2^{(K-1)/K}\leq\eta^\frac{1}{K}$ and $M=c(\VE N)^{1/k}$.\qed




\subsection{Decomposition}

In order to exploit the phenomenon exhibited in Lemma \ref{FourierError} in a proof of Theorem \ref{MainThm} one would naturally try, as we did in the proof of Theorem \ref{ergodic}, to make use of a decomposition theorem that will allow use to decompose $f$ into a structured component and a suitably uniform (anti-structuered) component.

It is easy to see that for any given $\eta>0$ and $f:\Z_N\to\C$ with $\|f\|_2\leq 1$, then the number of $\xi\in\Z_N$ such that $|\widehat{f}(\xi)|\geq\eta$ is at most $\eta^{-2}$, since $\|\widehat{f}\|_2\leq1$. Using the Fourier inversion formula together with this fact we certainly split $f=g+h$, where
\[g(x)= \sum_{\xi\in\Gamma} \widehat{f}(\xi)e(\xi x/N) \ \ \text{and} \ \ h(x)=\sum_{\xi\notin\Gamma} \widehat{f}(\xi)e(\xi x/N)\]
with $\Gamma=\{\xi\in\Z_N\,:\,|\widehat{f}(\xi)|\geq\eta\}$. It is then immediate that $|\widehat{h}(\xi)|<\eta$, and that $g$ is indeed ``structured'' in the sense that it involves only a bounded number of characters.

As we shall see below this simple decomposition will unfortunately be insufficient for our purpose. The problem being that we need a much stronger relationship between  $\|\widehat{h}\|_\infty$ and the upper bound on the size of $\Gamma$. The following result, which shows that one can indeed obtain this, modulo a small $L^2$-error, is part of the standard folklore of additive combinatorics (see for example \cite{G} Proposition 2.5).

\begin{propn}\label{U2} Let $f:\Z_N\to\C$ with $\|f\|_2\leq 1$. Write $\Z_N=\{\xi_1,\dots,\xi_N\}$ so that $|\widehat{f}(\xi_1)|\geq\cdots\geq |\widehat{f}(\xi_N)|$.

For every $\VE>0$ and $\eta:\N\to\R_+$ a positive function that decreases to 0, there exists $m=m(\VE,\eta)\in\N$ and a decomposition
\[f=f_1+f_2+f_3,\]
with 
\[f_1(x)=\sum_{j=1}^m \widehat{f}(\xi_j)e(\xi_jx/N)\]
while $\|\widehat{f_2}\|_\infty\leq\eta(m)$ and $\|f_3\|_2\leq\VE$.
\end{propn}

We note that the proof of this result, we include below for completeness, gives us the desired decomposition for some natural number $m$ that is bounded above by a number that results from starting with $1$ and applying the function $t\mapsto\eta(t)^{-2}$ at most $\VE^{-2}$ times.

\begin{proof}[Proof of Proposition \ref{U2}, \cite{G}] Choose an increasing sequence of positive integers $m_1,m_2,\dots$ with $m_1=1$ and $m_{r+1}\geq \eta(m_r)^{-2}$ for every $r$. We now choose $r$ and attempt to prove the result using the decomposition $f=f_1+f_2+f_3$ with
\[f_1(x)=\sum_{j\leq m_r} \widehat{f}(\xi_j) e(\xi_jx/N),\ \ f_2(x)=\sum_{j>m_{r+1}}\widehat{f}(\xi_j) e(\xi_jx/N),\ \ f_3(x)=\sum_{m_r<j\leq m_{r+1}}\widehat{f}(\xi_j) e(\xi_jx/N).\]
Then $f_1$ is a linear combination of at most $m_r$ characters. Also $\|\widehat{f_2}\|_\infty\leq m_{r+1}^{-1/2}\leq \eta(m_r)$, since, by Plancherel's theorem, there can be at most $m_{r+1}$ Fourier coefficients of $f$ whose magnitude is at least $m_{r+1}^{-1/2}$.

Therefore, we are done if $\|f_3\|_2\leq\eps$. But the possible functions $f_3$ (as $r$ varies) are disjoint parts of the Fourier decomposition of $f$, so at most $\VE^{-2}$ of them can have norm greater then $\VE$. Thus there exists $r\leq\VE^{-2}$ such that the proposed decomposition works.
\end{proof}

To find the appropriate function $\eta(m)$ we use the following result from the theory of diophantine approximation, the proof of this result is presented in Appendix B.

\begin{lem}[Main term estimate]\label{FourierMain}
There exists a constant $C_{k,\ell}$ such that for all $0<\VE<1/2$ we have
\eq \left|\left\{1\leq n< M\,:\, \|P_i(n)\xi_j/N\| < \VE \ \text{for all} \ 1\leq i\leq \ell,\,1\leq j\leq m\right\}\right|\geq (\eps/m)^{C_{k,\ell} m^2}M \ee
provided $M\geq(\eps/m)^{-C_{k,\ell}m^2}$, where $\|\A\|$ denotes, for each $\A\in\R$, the distance from  $\A$ to the nearest integer.
\end{lem}

\subsection{Proof of Theorem \ref{MainThm}}
Inserting the decomposition from Proposition \ref{U2}, for say $\eps/8$ and a positive decreasing function $\eta:\N\mapsto\R_+$ to be chosen later (using Lemma \ref{FourierMain}), into (\ref{intersectionF}) we obtain
\begin{align*}
\frac{|A\cap(A+P_i(n))|}{N}&=\frac{1}{N}\sum_{x\in\Z_N}f_1(x)f_1(x-P_i(n))+\frac{1}{N}\sum_{x\in\Z_N}(f_2+f_3)(x)g(x-P_i(n))\\
&\geq \left| \frac{1}{N}\sum_{x\in\Z_N}f_1(x)f_1(x-P_i(n)) \right|-\left| \frac{1}{N}\sum_{x\in\Z_N}(f_2+f_3)(x)g(x-P_i(n))  \right|
\end{align*}
for each $1\leq i\leq \ell$,
where $g:\Z_N\mapsto\R$ is some function that satisfies $\|g\|_2\leq 2$.

By the Cauchy-Schwarz inequality it follows that
\eq
\left|\frac{1}{N}\sum_{x\in\Z_N}f_3(x)g(x-P_i(n))\right|\leq \|f_3\|_2 \|g\|_2\leq \eps/4\ee
for all $n\in\N$, while from Lemma \ref{FourierError} it follows that for all but at most $\ell C_1\eta(m)^{1/K}M$ values of $n\in[1,M]$ we also have
\eq
\left|\frac{1}{N}\sum_{x\in\Z_N}f_2(x)g(x-P_i(n))\right|<\eta(m)^{1/K}
\ee
simultaneously for all $1\leq i\leq \ell$.

Using the fact that $\xi_1=0$ (a consequence of $f$ is non-negative) one can conclude that for each $1\leq i\leq \ell$ we have
\begin{align*}
\left|\frac{1}{N}\sum_{x\in\Z_N}f_1(x)f_1(x-P_i(n))\right|&=\left|\sum_{j=1}^m |\widehat{f}(\xi_j)|^2 e(\xi_j P_i(n)/N)\right|\\
&\geq\sum_{j=1}^m |\widehat{f}(\xi_j)|^2-\sum_{j=1}^m |\widehat{f}(\xi_j)|^2 \left|e(\xi_j P_i(n)/N)-1\right|\\
&\geq |\widehat{f}(\xi_1)|^2-\VE/2\\
&=\left(\frac{|A|}{N}\right)^2-\frac{\VE}{2}
\end{align*}
provided that $|e(\xi_j P_i(n)/N)-1|\leq \eps/2$ for all $1\leq j\leq m$.
It therefore follows from Lemma \ref{FourierMain}, since $|e(x)-1|\leq 2\pi\|x\|$, that for at least $(\eps/4\pi m)^{C_{k,\ell} m^2}M$ values of $n\in[1,M]$ the main terms will satisfy
\[\left|\frac{1}{N}\sum_{x\in\Z_N}f_1(x)f_1(x-P_i(n))\right|\geq \left(\frac{|A|}{N}\right)^2-\frac{\VE}{2}\]
simultaneously for all $1\leq i\leq \ell$, provided $M\geq(\eps/4\pi m)^{-C_{k,\ell}m^2}$.

Therefore, if we choose $\eta(t):=(\ell C_1)^{-K} (\eps/4\pi t)^{C_{k,\ell}Kt^2}/2$, it follows that
$\ell C_1\eta(m)^{1/K} <(\eps/4\pi m)^{C_{k,\ell} m^2}$,
thus establishing Theorem \ref{MainThm} with
\[N_1(\VE,P_1,\dots,P_\ell)=C(\eps/4\pi m)^{-kC_{k,\ell}m^2-1}\]
where the constant $C$ depends only on the polynomials $P_1,\dots,P_\ell$.

Since, as we remarked above, $m$ is bounded above by a number that results from starting with $1$ and applying the function $t\mapsto\eta(t)^{-2}$ at most $\VE^{-2}$ times, $N_1$ will clearly be a tower type bound with height proportional to $\VE^{-2}$. To be more precise, it is not hard to verify that
\[(\eps/4\pi m)^{-kC_{k,\ell}m^2-1}=T(C^*_{k,\ell}+C\eps^{-2}),\] where $T(1)=1,\ T(j+1)=2^{T(j)}$ is the tower function.

Finally we remark that if one uses explicit Vinogradov type bounds for the Weyl sums, then one obtains the constant $C_{k,\ell}=C k^C \ell^C$ in Proposition \ref{nt} and hence $C_{k,\ell}^*\leq \log^* (k) +\log^* (\ell) +C^*$ where $\log^*$ denotes the inverse of the tower function $T(n)$.
\qed



\section{Simultaneous Linear Recurrence}\label{johnsec}

In this section we will give a combinatorial proof of Theorem \ref{John}, shown to us by John Griesmer \cite{John}.


\subsection{A combinatorial proof of Khintchine's theorem}

We first establish the result for $\ell=1$ (which of course corresponds to the case $m=0$). This is a quantitative formulation of the $\VE$-optimal extension of Poincar\'e's recurrence theorem, due to Khintchine, that was discussed in the introduction.

\begin{lem}[Bergelson \cite{B2}]\label{khinchin}
Let $(X,\MM,\mu,T)$ be an invertible measure preserving system and $A\in\mathcal{M}$. For every $\VE>0$ and $B\subseteq\N$ with
$|B|\geq \VE^{-1}$, there exists a non-zero $n\in B-B$  such that
\eq
\mu(A\cap T^{-n} A)>\mu(A)^2-\eps.
\ee
\end{lem}

\begin{proof}
Let $\VE>0$ and $v_1,\dots,v_N$ be distinct natural numbers with $N\geq\VE^{-1}$.
Since
\[N\mu(A)=\int_X \sum_{j=1}^N 1_A(T^{v_j}x)\,d\mu\]
it follows from the Cauchy-Schwarz inequality that
\[N^2\mu(A)^2\leq\int_X \Bigl(\sum_{j=1}^N 1_A(T^{v_j}x)\Bigr)^2\,d\mu\]
and hence that
\[\frac{1}{N^2}\sum_{1\leq j,k\leq N}\mu(T^{-v_j}A\cap T^{-v_k}A)\geq \mu(A)^2.\]
It then follows from the fact that there are only $N$ diagonal terms that there must exist a pair $1\leq j<k\leq  N$ for which
\[\mu(A\cap T^{-(v_k-v_j)}A)\geq \mu(A)^2-\VE.\qedhere\]
\end{proof}

\subsection{Proof of Theorem \ref{John}}

Let $m$ be a non-negative integer. It clearly suffices to prove the Theorem for $\ell=2^m$. We proceed by induction on $m$, noting that Lemma \ref{khinchin} above covers the base case, namely when $m=0$, since we can simply apply that result to the transformation $T^{c_1}$ in place of $T$ for any given $c_1\in\Z\setminus\{0\}$.

We now assume that the result holds for a given non-negative integer $m$, fix $c_1,\dots,c_{2^{m+1}}\in\Z\setminus\{0\}$ and distinct natural numbers $v_1,\dots,v_N$ with $N\geq2^{2L}$ where $\log_2^mL\geq C\VE^{-1}$.

We now define a coloring of the edges of the complete graph $K_N$ on the vertices $v_1,\dots,v_N$. Color the edge between vertices $v_j$ and $v_k$ \emph{red} if
\[\mu(A\cap T^{-c_i(v_j-v_k)} A)>\mu(A)^2-\eps\]
for all $1\leq i\leq 2^m$, and color it \emph{blue} if this is not the case.

Since $N\geq 2^{2L}$ and $\log_2^mL\geq C\VE^{-1}$, it follows from Ramsey's theorem and the inductive hypothesis that $K_N$ must contain a complete \emph{red} subgraph with at least $L$ vertices. In other words, there exists a sub-collection
\[\{w_1,\dots,w_L\}\subseteq\{v_1,\dots,v_N\}\]
such that \emph{for every} $1\leq j,k\leq L$ we have
\[\mu(A\cap T^{-c_i(w_j-w_k)} A)>\mu(A)^2-\eps\]
for all $1\leq i\leq 2^m$. Applying the inductive hypothesis once more, this time with $c_{2^m+1},\dots,c_{2^{2m}}$ to the collection $w_1,\dots,w_L$ of distinct natural numbers obtained above, it follows that there necessarily exists a pair $1\leq j<k\leq L$ such that
\[\mu(A\cap T^{-c_i(w_j-w_k)} A)>\mu(A)^2-\eps\]
for all $2^m+1\leq i\leq 2^{2m}$. This completes the proof.\qed


\section{The proof of Theorem \ref{LM1}}\label{liftsec}
It is easy to see that in order to prove Theorem \ref{LM1} it suffices to establish the following reformulation.

\begin{thm}\label{1}
Let $k\geq 2$ and $P_1,\ldots,P_\ell \in\Z[n]$ with $P_i(0)=0$ and $\deg P_i\leq k$ \emph{for all} $1\leq i\leq \ell$.

If $A\subseteq[1,N]$ and
$\{P_1(n),\dots ,P_\ell (n)\}\nsubseteq A-A$
for any $n\ne0$,
then we necessarily have
\[\frac{|A|}{N}\leq C\left(\frac{\log\log N}{\log N}\right)^{1/\ell(k-1)}\]
for some absolute constant $C=C(P_1,\dots,P_\ell )$.
\end{thm}
As remarked in Section \ref{outline}, Theorem \ref{1} was established for families of linearly independent polynomials in \cite{LM1}.
In the case of a single polynomial ($\ell=1$), this result has originally obtained by Lucier \cite{Lucier} (with slightly weaker bounds) and, to the best of our knowledge, constitutes the best bounds that are currently known for arbitrary polynomials with integer coefficients and zero constant term.

A simple modification of the \emph{lifting} argument utilized in \cite{LM1} will enable us to deduce Theorem \ref{1} as a corollary of the following higher dimensional result, which was the main result in \cite{LM1} (see also \cite{LM1'}).

\begin{thm}[Lyall and Magyar \cite{LM1}, \cite{LM1'}]\label{2}
If $B\subseteq[1,N]^k$ and $(n,n^2,\dots,n^k)\notin B-B$ for any $n\ne0$ then we necessarily have
\[\frac{|B|}{N^k}\leq C\left(\frac{\log\log N}{\log N}\right)^{1/(k-1)}\]
for some absolute constant $C=C(k)$.
\end{thm}

As in \cite{LM1}, we again speculate that the methodology of Balog et al. \cite{BPPS} may be applied in this higher dimensional situation to obtain far superior bounds in Theorem \ref{2} and hence also in Theorem \ref{1}.

\subsection{Proof that Theorem \ref{2} implies Theorem \ref{1}}\label{2.1}

Let $P_i(n)=c_{i1} n+\cdots+c_{ik}n^k$ for $1\leq i\leq\ell$.

Suppose that the coefficient matrix $\mathcal{P}=\{c_{ij}\}$ has rank $r$ with $1\leq r\leq \ell$. Without loss of generality we will make the additional assumption that it is in fact the first $r$ polynomials $P_1,\dots,P_{r}$ that are linearly independent and use $\mathcal{R}$ to denote the $r\times k$ matrix corresponding to the first $r$ rows of $\mathcal{P}$.
As a consequence of this assumption it follows that the remaining polynomials, $P_{r+i}$ with $1\leq i\leq \ell-r$, can be expressed as
\[P_{r+i}=d_{i1}P_1+\cdots+d_{ir} P_{r}\]
where $\mathcal{D}=\{d_{ij}\}$ is some $(\ell-r)\times r$ matrix with rational coefficients.
Note that
\begin{align*}
&\mathcal{P}:\Z^k\rightarrow\Z^\ell \\
&\mathcal{R}:\Z^k\rightarrow\Z^r \\
&\mathcal{D}:\mathcal{R}(\Z^k)\rightarrow\Z^{\ell-r}
\end{align*}
and
\[\mathcal{P}(b)=\left(\begin{matrix} \mathcal{R}(b) \\ \mathcal{D}(\mathcal{R}(b))\end{matrix}\right).\]

Let $A^\ell=A\times\cdots\times A\subseteq[1,N]^\ell$ and set $\D=|A|/N$.
The full rank assumption on the matrix $\mathcal{R}$ ensures that there exists an absolute constant $c$, depending only on the coefficients of the matrix $\mathcal{R}$, such that
\[\bigl|\mathcal{R}(\Z^k)\cap(A^r-s)\bigr|\geq c \D^r N^r\]
for some $s\in[1,c^{-1}]^r$.
Thus, if we choose $N'$ to be a large enough multiple of $N$ (again depending only the coefficients of the matrix $\mathcal{R}$) and let
\[B'=\left\{b\in[-N',N']^k\,:\,\mathcal{R}(b)\in A^r -s\right\},\]
it follows that
\[|B'|\geq c\, \D^r  N^k.\]

Since \[\sum_{t\in\Z^{\ell-r}}\sum_{b\in B'}1_{A^{\ell-r}}(\mathcal{D}(\mathcal{R}(b))+t)=|A|^{\ell-r}|B'|\]
it follows that there exists $c=c(\mathcal{P})$ and $t\in\Z^{\ell-r}$ such that
\[\left|\left\{b\in B'\,:\, \mathcal{D}(\mathcal{R}(b))\in A^{\ell-r}-t\right\}\right|\geq c\delta^{\ell-r}|B'|.\]

Hence, if we let
\[B=\left\{b\in[-N',N']^k\,:\,\mathcal{P}(b)\in A^\ell -m\right\},\]
where $m=(s,t)\in\Z^\ell$,
it follows that
\[|B|\geq c\, \D^\ell  N^k.\]

Theorem \ref{1} now follows from Theorem \ref{2} since if there were to exist an $n\ne0$ such that
\[(n,n^2,\dots,n^k)\in B-B\]
this would immediately implies that
\[(P_1(n),\dots,P_\ell (n))\in A^\ell-A^\ell,\]
since $\mathcal{P}(B)\subseteq A^\ell -m$.\qed


\begin{appendix}\label{nt1}
\section{Counting solutions to systems of polynomial diophantine equations}

The aim of this section is to prove Proposition \ref{nt}. We do this by showing that it follows easily from counting the integer solutions $1\leq x_1,\ldots,x_K,y_1,\ldots,y_K\leq M$ of the system of equations
\eq x_1^i+\cdots+x_K^i=y_1^i+\cdots +y_K^i\ee
where the exponent $i$ ranges from 1 to $k$, known as Tarry's problem.  An asymptotic formula for the number of solutions $J_{K,k}(M)$ (as $M\to\infty$) was originally obtained by Hua \cite{H}, see also Wooley \cite{W}. In particular, it follows from these results that
\eq J_{K,k}(M)\,\leq C_kM^{2K-k(k+1)/2}\ee
as long as $K>Ck^2\log\,k$. For additional discussion of Tarry's problem, see \cite{ACK} and \cite{W}.

\begin{proof}[Proof of Proposition \ref{nt}] Let $P(x)=c_k x^k+\dots +c_1 x$ be an integral polynomial. For given
\[1\leq x_1,\ldots,x_K,y_1,\ldots,y_K\leq M\]
and for $1\leq i\leq k$, let
\eq\label{5.3}
 s_i = x_1^i+\cdots +x_K^i\,-\,y_1^i-\cdots -y_K^i\ee

Then $x_1,\ldots,x_K,y_1,\ldots,y_K$ is a solution of equation (\ref{2.5}) if and only if $c_1 s_1+\cdots +c_k s_k=0$.

For given $s_1,\ldots,s_k,$ let $J_{K,k}(M;s)$ denote the number of integer solutions of the system (\ref{5.3}). Then, as usual, one can express the number of solutions as a multiple integral of the form
\[J_{K,k}(M;s)=\int_0^1\ldots\int_0^1 |S(\te_1,\ldots,\te_k)|^{2K}\,e^{-2\pi i (s_1\te_1+\cdots +s_k\te_k)}\ d\te_1\ldots d\te_k\]
where
\[S(\te_1,\ldots,\te_k) = \sum_{m=1}^M e^{2\pi i (m\te_1+\cdots+m^k\te_k)}.\]

Therefore, we have that $J_{K,k}(M;s)\leq J_{K,k}(M)$ uniformly in $s_1,\ldots,s_k$. For a solution of (\ref{2.5}) the values $s_1,\ldots,s_{k-1}$ determine $s_k$, and since $|s_i|\leq kN^i$ for $1\leq i\leq k-1$, one estimates the number of solutions of (\ref{2.5}) from above by $C_k\,M^{(k-1)k/2}M^{2K-k(k+1)/2} = C_kM^{2K-k}$.
This proves Proposition \ref{nt}.
\end{proof}


\section{Simultaneous polynomial diophantine approximation}\label{nt2}

The purpose of this section is to supply a proof of Lemma \ref{FourierMain}. We in fact establish the following more general result.

\begin{propn}\label{newlemma} Let $P_1,\ldots,P_\ell \in\Z[n]$ with $P_i(0)=0$ and $\deg P_i\leq k$ \emph{for all} $1\leq i\leq \ell$ and $\te_1,\ldots,\te_m\in\R$. Then for any $0<\VE\leq 1/2$ and $N\in\N$ we have that
\eq
|\{1\leq n\leq N\,:\,\|P_i(n)\te_{i'}\|<\eps \ \text{for all} \ 1\leq i\leq \ell,\, 1\leq i'\leq m\}|\geq (\VE/d)^{C_k d^2}\ee
where $d=k\ell m$ and $C_k>0$ is a constant depending only on $k$.
\end{propn}

It is easy to see that Proposition \ref{newlemma} is an immediate corollary of the following result.

\begin{propn}\label{vector} Let $d_1,\ldots,d_k\in\N$ and $\al_j\in\R^{d_j}$ for $1\leq j\leq k$. There exists $C_k>0$ such that for any $0<\VE\leq 1/2$ and $N\in \N$ we have
\eq\label{36}
 |\{1\leq n\leq N\,:\, \|n^j\al_j\|<\eps \ \text{for all} \ 1\leq j\leq k\}|\geq (\VE/d)^{C_k d^2}\ee
where $d=d_1+\cdots +d_k$ and $\|\al_j\|$ denotes, for all $\A_j\in\R^{d_j}$, the distance from $\A_j$ to the nearest integer point in $\Z_{d_j}$.
\end{propn}

\begin{proof}[Proof of Proposition \ref{newlemma}]
Let $P_i(n)=\sum_{j=1}^k c_{ij} n^j$ for $1\leq i\leq k$. For any given $1\leq i\leq \ell$ and $1\leq i'\leq m$ we of course have $\|P_i(n)\te_{i'}\|<\VE$ whenever $\|n^j c_{ij}\te_{i'}\|<\VE/k$ for all $1\leq j\leq k$.
Thus, if we apply Proposition \ref{vector} to the vectors $\A_1,\dots\A_k\in\R^{\ell m}$ where $\al_j= (c_{ij}\te_{i'})_{1\leq i\leq \ell,\,1\leq i'\leq m}$ for  $1\leq j\leq k$, then Proposition \ref{newlemma} follows.
\end{proof}

We are therefore reduced to the task of proving Proposition \ref{vector}.
The special case $k=2$ and $\al_1=0$ is precisely Proposition A.2. in \cite{GT}. In fact, their argument generalizes to our case in a straightforward manner and as such we will sketch only the main steps and refer to the proofs in \cite{GT}.

\subsection{The proof of Proposition \ref{vector}}

Let $\La=\La_1\times\cdots\times\La_k$ where each $\La_i\subs\R^{d_i}$ is a full rank lattice. Recall, from \cite{GT}, that the theta function
\eq\label{theta}
 \Te_{\La}(t,x):=\sum_{m\in\La} e^{-\pi t|x-m|^2} = \frac{1}{t^{D/2} \det (\La)} \sum_{\xi\in\La^*} e^{-\pi |\xi|^2/t}\,e(\xi\cdot x)\ee
where $\La^*:=\{\xi\in\R^d\,:\, \xi\cdot m\in\Z \ \text{for all} \ m\in\La\}$ is the dual lattice of $\La$, from which it follows that
\eq\label{F}
 F_{\La,\al}(N):= \det (\La)\frac{1}{N}\sum_{n=1}^N\Te_\La (1,n\circ\al)=\sum_{\xi\in\La^*} e^{-\pi |\xi|^2}\frac{1}{N}\sum_{n=1}^N e(\xi\cdot (n\circ\al))\ee
where $\al=(\al_1,\ldots,\al_k)$ and for each $n\in\N$ we define $n\circ \al := (n\al_1,n^2\al_2,\ldots,n^k\al_k)$.
As in \cite{GT} we also define the quantity
\eq A_\La := \det(\La)\ \sum_{m\in\La} e^{-\pi |m|^2} = \sum_{\xi\in\La^*} e^{-\pi |\xi|^2}.\ee

The crucial ingredient in the proof of Proposition \ref{vector} is the following lower bound  on $F_{\La,\al}(N)$, whose proof we outline in Section \ref{proofoflower}.

\begin{propn}\label{lower bound} Let $\La=\La_1\times\cdots\times\La_k$ where $\La_i\subs\R^{d_i}$ ($d_i\geq 0$) is a full rank lattice, such that $\det (\La)\geq 1$. Then for all $\al=(\al_1,\ldots,\al_k)$, with $\al_i\in\R^{d_i}$ and $N\in\N$ one has
\eq F_{\La,\al}(N) \geq (Cd)^{-C_kd^2} A_\La^{-C_k d}\ee
where $d=d_1+\cdots +d_k\geq 1$.
\end{propn}

Note that in the statement we allow the degenerate case $d_i=0$, when $\La_i=\R^{d_i}=\{0\}$.

Assuming Proposition \ref{lower bound} we can now establish Proposition \ref{vector} as in \cite{GT}.

\begin{proof}[Proof of Proposition \ref{vector}]
 For simplicity of notation we set $c_d:=(Cd)^{-C_kd^2}$.  Let $\VE>0$ and $\La:= (R\Z)^d$ where $R:=C\,C_k d^2\VE^{-2}$ (with a suitable large constant $C$). Note that
\eq\label{41}
 A_\La = R^d \sum_{m\in\Z^d} e^{-\pi R^2 |m|^2} \leq (10R)^d.\ee
If $\|n\circ\al\|\geq\VE$ that is $|n\circ\al-m|\geq\eps$ for all $m\in\Z^d$, then for $\be=R\al$ one has that \[|n\circ\be-Rm|^2\geq R^2\VE^2/2+|n\circ\be-Rm|^2/2\] for all $m\in\Z^d$. Thus by (\ref{theta}) and the choice of $R$ it follows that
\begin{align*}
\det (\La) \Te_{\La}(1,n\circ\be)&\leq e^{-\pi R^2\VE^2/2} \det (\La) \Te_{\La}(1/2,n\circ\be)\\
&\leq e^{-\pi R^2\VE^2/2}  2^{d/2} \sum_{\xi\in\La^*} e^{-\pi |\xi|^2}\\
&\leq \frac{c_d}{2} A_\La^{-C_k d}.
\end{align*}
If one defines the set $G:= \{n\in [1,N]\,:\, \|n\circ\al\|\leq\VE\}$, then by (\ref{F}) and (\ref{41}) it follows that
\[\frac{1}{N} \sum_{n\in G} \det (\La) \Te_{\La}(1,n\circ \be) \geq \frac{c_d}{2}\ A_\La^{-C_k d}.\]
Also
\[\det (\La)\Te_{\La}(1,n\circ \be)\leq A_\La\leq (C_k d^2\VE^{-2})^d\]
for all $n$. This implies (\ref{36}) since
\[\frac{|G|}{N} \geq \frac{c_d}{2}\ A_\La^{-C_d-1} \geq (Cd)^{-C_kd^2} (C_k d^2\VE^{-2})^{-C_k d^2} \geq (\VE/d)^{C_k' d^2}.\qedhere\]
\end{proof}

\subsection{Three Lemmas}
In order to prove Proposition \ref{lower bound} we need the the following three lemmas, which we present without proof. These lemmas correspond to Lemmas A.5, A.6 and A.7 in \cite{GT} and are proven in essentially the same way. We leave this for the interested reader to verify.

\begin{lem}[Properties of $F_{\La,\A}$]\label{P} Let $\La=\La_1\times\cdots\times\La_k$, where each $\La_i\subs\R^{d_i}$ is a full rank lattice.
Let $\al=(\al_1,\ldots,\al_k)$, with $\al_i\in\R^{d_i}$ and $N>20$.
\begin{itemize}
\item[(i)] For any $c\in (\frac{1}{10},1)$, we have $\ F_{\La,\al}(N)\geq \frac{c}{2} F_{\La,\al}(cN)$.
\item[(ii)] For any integer $1\leq q\leq \frac{N}{2}$, we have $\ F_{\La,\al}(N)\geq \frac{1}{2q} F_{\La,\al}(\frac{N}{q})$.
\item[(iii)] Let $0<\eps\leq\frac{1}{d}$. If $\be=(\be_1,\ldots,\be_k)$ such that $|\be_i-\al_i|\leq \eps N^{-i}$ for $1\leq i\leq k$, then for all $1\leq n\leq N$ we have
\[\Te_\La (1,n\circ\al)\geq c^k\,\Te_{(1+\eps)\La} (1,n\circ (1+\eps)\be)\]
and hence
\[F_{\La,\al}(N)\geq c^k\,F_{(1+\eps)\La,(1+\eps)\be}(N).\]
\end{itemize}
\end{lem}

\begin{lem}[Schmidt's alternative]\label{S} Let $\La=\La_1\times\cdots\times\La_k$, where each $\La_i\subs\R^{d_i}$ is a full rank lattice and let $N> (4A_\La)^{C_k}$. One of the following two alternatives holds:
\begin{itemize}
\item[(i)] $\ F_{\La,\al}(N)\geq 1/2;$
\item[(ii)] There is a positive integer $q \leq d (4A_\La)^{C_k}$, and a primitive $\xi_i\in\La_i^*\backslash\{0\}$, such that
\eq |\xi_i|\leq C(\sqrt{d}+\sqrt{\log A_\La})\ee
and
\eq \|q\xi_i\cdot\al_i\|_{\R/\Z} \leq(4A_\La)^{C_k} N^{-i}\ee
\end{itemize}
\end{lem}

Recall that $\xi\in\La_i^*$ is \emph{primitive} if $\xi/n\notin\La_i^*$ for any integer $n\geq 2$.

\begin{lem}[Descent]\label{D} Suppose that $\La'\subseteq\R^{d-1}$ and $\La\subseteq\R^d$ are full rank lattices with $\La'\subs\La$, where $R^{d-1}$ is regarded as a subset of $\R^d$. Suppose that $\al'\in\R^{d-1}$, that $\al\in\R^d$ and that $\al-\al'\in\La$. Then
\eq F_{\La,\al}(N)\,\geq\,\frac{\det (\La)}{\det (\La')}\,F_{\La',\al'}(N)\ee
\end{lem}

The only substantial difference from \cite{GT} in this section is in the proof of Lemma \ref{S}, where we need estimates for the exponential sums, defined for
$\te=(\te_1,\ldots,\te_k)\in\R^k$ by
\eq S_N(\te) =\frac{1}{N}\sum_{n=1}^N e(n\te_1+\cdots +n^k \te_k).\ee
The following is precisely what is required.
\begin{lem}[Weyl Inequality] Let $0<\de\leq1/2$. There exist a positive constant $C_k>0$, such that if $N\geq\de^{-C_k}$ and
\eq |S_N(\te)|\geq\de\ee
then there exists a positive integer $q\leq \de^{-C_k}$ such that
\eq \|q\te_i\| < \de^{-C_k} N^{-i}\ee
for all $1\leq i\leq k$.
\end{lem}
This formulation follows easily form standard estimates on Weyl sums, see for example \cite{LM1}, Lemma 5. In fact using the sophisticated estimates of Vinogradov, one may take $C_k=C\,k^2\log\,k$, however for simplicity we do not develop such bounds here.

\subsection{The proof of Proposition \ref{lower bound}}\label{proofoflower}
As in \cite {GT}, the proof of Proposition \ref{lower bound} will follow, via an iteration, from the following result.

\begin{propn}[Inductive lower bound on $F_{\La,\al}$]\label{induct}
Suppose $\al=(\al_1,\ldots,\al_k)$, $\La=\La_1\times\cdots\times\La_k$ such that $\al_i\in\R^{d_i}$ and $\La_i\in\R^{d_i}$ is a full rank lattice. Let $N>(4A_\La)^{C_k}$ be an integer. Then either $F_{\La,\al}(N)\geq 1/2$ or there is an $\al'\in\R^{d-1}$ and a full rank lattice $\La'=\La_1'\times\cdots\times\La_k'\subs\R^{d-1}$ with
\eq\label{49} A_{\La'}\leq C(\sqrt{d}+\sqrt{\log A_\La})\,A_\La\ee
and an $N'\geq d^{-C} (4A_\La)^{-C_k}N$ such that
\eq\label{50} F_{\La,\al}(N)\geq d^{-C} (4A_\La)^{-C_k} F_{\La',\al'}(N')\ee
\end{propn}

\begin{proof}[Proof of Proposition \ref{induct}] Assuming $F_{\La,\al}(N)< 1/2$ and applying Lemma \ref{S}, there exists an $1\leq i\leq k$, a primitive $\xi_i\in\La_i^*\backslash\{0\}$ and a positive integer $q\leq (4A_\La)^{C_k}$, such that
\[\|\xi_i\cdot q^i\al_i\| \leq (4A_\La)^{C_k} N^{-i}\] (by changing the value of the constant $C_k$). Fixing the lattice $\La_i$, and arguing as in Proposition A.8 of \cite{GT}, it follows that there is a $\be_i\in\R^{d_i}$ such that
\eq\label{51}
|\beta_i-q^i\al_i|\leq (4A_\La)^{C_k} N^{-i}\ee
and an $m_i\in\La_i$ such that $\be'_i=\be_i-m_i\in (\R\xi_i)^\bot\simeq\R^{d_i-1}$. Let $N_*=cd^{-C}(4A_\La)^{-C_k} N$, then by the choice of $N_*$, we have
\eq |\be_i-q^i\al_i|\leq d^{-1} N_*^{-i}\ee
and moreover by Lemma \ref{P} (i) and (ii) it follows that
\[F_{\La,\al}(N) \geq d^{-C} (4A_\La)^{-C_k} F_{\La,\al}(N_*) \geq  d^{-C}(4A_\La)^{-C_k} F_{\La,q\circ\al}(N_*/q)\]
If $\be=(\be_1,\ldots,\be_k)$ with $\be_i$ satisfying (\ref{51}), and $\be_j:=q^j\al_j$ for each $j\neq i$, then by Lemma \ref{P} (iii) with $\eps=1/d$ we have
\eq F_{\La,\al}(N) \geq d^{-C}(4A_\La)^{-C_k}F_{\La,q\circ\al}(N_*/q)\geq d^{-C} (4A_\La)^{-C_k} F_{(1+\eps)\La,(1+\eps)\be}(N')\ee
where $N'=N_*/q$. Note that by the choice of $N$ and the upper bound on $q$, $N'$ satisfies the claimed lower bound.
Finally let $\al'=(1+\eps)\be'$, where $\be_j'=\be_j$ for $j\neq i$, and let $\La'=\La_1'\times\cdots\times\La_k'$ such that $\La_i'=(1+\eps)\La_i\cap (\R\xi_i)^\bot$ and $\La_j':=(1+\eps)\La_j$ for $j\neq i$.  From Lemma \ref{D} using the facts that $\det(\La)=\prod_i \det(\La_i)$ and $(1+1/d)^d\leq e$, one obtains
\[F_{\La,\al}(N) \geq d^{-C} (4A_\La)^{C_k}\frac{\det (\La_i)}{\det (\La_i\cap (\R\xi_i)^\bot)}F_{\La',\al'}(N').\]
The rest of the argument goes exactly as in \cite{GT}, one estimates
\[\frac{\det (\La_i)}{\det (\La_i\cap (\R\xi_i)^\bot)} = \frac{A_{\La_i}}{A_{\La_i\cap (\R\xi_i)^\bot)}}\leq |\xi_i|^{-1}\leq C(\sqrt{d}+\sqrt{\log A_{\La_i}}).\]
Since $A_\La=\prod_i A_{\La_i}$ and in particular $A_{\La_i}\leq A_\La$, the claimed bounds (\ref{49}) and (\ref{50}) follow.
\end{proof}

Note that if $d_i=1$ the $(\R\xi_i)^\bot=\{0\}$, thus $\La_i'=\{0\}$, $d_i'=0$ and $\Te_{\La_i'}(t,x)\equiv 1$. We allow this to avoid the need to discuss separate cases. Iterating this proposition leads to the claimed lower bound on $F_{\La,\al}(N)$ in Proposition \ref{lower bound}, in a straightforward manner, as in the proof of Proposition A.9 in \cite{GT}.


\begin{proof}[Proof of Proposition \ref{lower bound}]
By the trivial lower bound
\[F_{\La,\al}(N)\geq \det (\La)/N\geq 1/N\] one may assume $N>d^{Cd^2}(4A_\La)^{C_k d}$ for some suitably large constants $C$ and $C_k$. Set $\La_0=\La$, $\al_0=\al$, $N_0=N$. Applying Proposition \ref{induct} repeatedly one obtains vectors $\al^{(j)}\in\R^{d-j}$, lattices $\La^{(j)}\subs\R^{d-j}$ and integers $N^{(j)}$. Thus there must exist a $j\leq d$ such that $F_j:=F_{\La^{(j)},\al^{(j)}}(N^{(j)})\geq 1/2$ (if $j=d$, then $\La^{(j)}=\{0\}$ hence $F_j=1$).

To check the validity of the iteration, since $A_\La\geq 1$, one may use the crude bound
\[\sqrt{d}+\sqrt{\log X}\leq CdA_\La^{1/d}\]
 for $X\geq1$, thus by (\ref{49}) one has
\[A_{\La^{(j)}} \leq (Cd)^d A_\La^C\]
for all $1\leq j\leq d$. This implies
\[N^{(j)}\geq d^{-C} (4A_{\La^{(j)}})^{-C_k} N^{(j-1)} \geq (Cd)^{-C_k d}A_\La^{-C_k}\]
thus by the choice of $N$, we have that $N^{(j)}\geq (4A_{\La^{(j)}})^{C_k}$ for all $1\leq j\leq d$, and hence by (\ref{50})
\[F_{j+1}\geq d^{-C} (4A_{\La^{(j)}})^{-C_k}F_j \geq  (Cd)^{-C_k d}A_\La^{-C_k}F_j\]
which gives the desired lower bound for $F_{\La,\al}(N)$.
\end{proof}

\end{appendix}


\end{document}